\documentclass[12pt,reqno]{amsart}
\usepackage{amssymb}
\usepackage{amsmath}
\usepackage{amsthm}
\usepackage{amsfonts}
\usepackage{tikz}
\usepackage{float}

\usepackage[colorlinks=true,
linkcolor=webgreen,
filecolor=webbrown,
citecolor=webgreen]{hyperref}

\definecolor{webgreen}{rgb}{0,.5,0}
\definecolor{webbrown}{rgb}{.6,0,0}

\newcommand{\seqnum}[1]{\href{https://oeis.org/#1}{\rm \underline{#1}}}
\newcommand{\NN}{\mathbb{N}}
\newcommand{\RR}{\mathbb{R}}

\DeclareMathOperator{\black}{black}
\DeclareMathOperator{\white}{white}
\DeclareMathOperator{\oD}{oddDisp}

\newcommand{\rFq}[5]{{}_{#1}F_{#2}\!\left(\begin{smallmatrix}#3\\#4\end{smallmatrix};#5\right)}

\begin{document}

\theoremstyle{plain}
\newtheorem{theorem}{Theorem}
\newtheorem{corollary}[theorem]{Corollary}
\newtheorem{lemma}[theorem]{Lemma}
\newtheorem{proposition}[theorem]{Proposition}

\theoremstyle{definition}
\newtheorem{definition}[theorem]{Definition}

\theoremstyle{remark}
\newtheorem{remark}[theorem]{Remark}

\title{Black-White Cell Capacity in $k$-ary Words and Permutations}

\author[SELA FRIED]{Sela Fried$^*$}
\thanks{$^*$Department of Computer Science, Israel Academic College,
52275 Ramat Gan, Israel.
\\
\href{mailto:friedsela@gmail.com}{\tt friedsela@gmail.com}}

\begin{abstract}
We introduce a new bargraph statistic that we call black-white cell capacity. It is obtained by coloring the cells of the bargraph in a chessboard style and recording the numbers of black and white cells contained in the bargraph. We study two word families under this statistic: $k$-ary words and permutations. We obtain the corresponding generating function, in the $k$-ary words case, and a closed-form formula for each $n$, in the permutations case. Of special interest are words containing an equal number of black and white cells, that we call bw-balanced. We obtain generating functions, closed-form formulas, and asymptotics in both cases.  
\end{abstract}

\maketitle

\section{Introduction}
For a natural number $m$ we denote by $[m]$ the set $\{1,2,\ldots,m\}$. A word of length $n$ is a sequence $u= u_1\cdots u_n$ of natural numbers. Each of the $u_i$s is naturally referred to as a letter of $u$. If $k\in\NN$ is such that $\max\{u_1,\ldots, u_n\}\leq k$, the word $u$ is called $k$-ary. The set of all $k$-ary words of length $n$ is denoted by $[k]^n$. A word of length $n$ consisting of the numbers $1,2,\ldots,n$, each appearing exactly once, is called a permutation of $[n]$. The set of all permutations of $[n]$ is denoted by $S_n$. Every word has a bargraph representation obtained by assigning each letter $u_i$ a column of cells of height $u_i$ (see Figure \ref{f1}). 

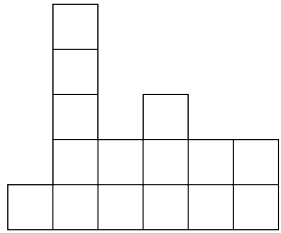
\begin{figure}[H]
\centering
\begin{tikzpicture}[scale=0.6]
  \def\size{7}
  \def\barpath{(0,0) -- (0,1) -- (1,1) -- (1,5) -- (2,5) -- (2,2) -- (3,2) -- (3,3) -- (4,3) -- (4,2) -- (5,2) -- (6,2) -- (6,0) -- cycle}
  \begin{scope}
    \clip \barpath;    
    \draw[thin] (0,0) grid (\size+1,\size+1);
  \end{scope}
  \draw[thin, black] \barpath;
\end{tikzpicture}
\caption{The bargraph of the word $152322$.}
\label{f1}
\end{figure}

The concept of representing a word as a bargraph allows analyzing words from a geometrical perspective. Many such word statistics, that make more sense in such a representation, have been systematically studied, for example, water cells, shedding light cells, and interior vertices, to name a few of the more esoteric ones. See Mansour and Shabani \cite{Man} for a survey on the subject.

In this work we introduce a seemingly new bargraph statistic, that we call black-white cell capacity. In plain words, we color the cells of the bargraph in a chessboard style, such that the southwestern cell is black, and count the number of black and white cells that it contains (see Figure \ref{f2}). 

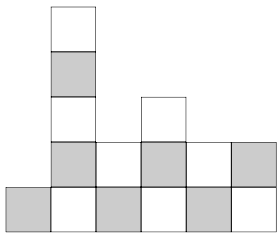
\begin{figure}[H]
\centering
\begin{tikzpicture}[scale=0.6]

  \def\barpath{(0,0) -- (0,1) -- (1,1) --
               (1,5) -- (2,5) -- (2,2) -- (3,2) --
               (3,3) -- (4,3) -- (4,2) --
               (5,2) -- (6,2) -- (6,0) -- cycle}

  \begin{scope}
    \clip \barpath;

    \foreach \x in {0,...,6} {
      \foreach \y in {0,...,5} {
        \pgfmathtruncatemacro{\parity}{mod(\x+\y,2)}
        \ifnum\parity=0
          \fill[gray!40,draw=none] (\x,\y) rectangle ++(1,1);
        \else
          \fill[white,draw=none]   (\x,\y) rectangle ++(1,1);
        \fi
      }
    }

    \draw[thin] (0,0) grid (6,5);
  \end{scope}

\end{tikzpicture}

\caption{The bargraph of the word $152322$ contains $7$ black cells and $8$ white cells.}
\label{f2}

\end{figure}

The following definition makes this more precise.

\begin{definition}
Let $i,h\in\NN$. We set 
\begin{align*}
\black_i(h) &= 
\begin{cases}
\lceil h/2\rceil,    & \textnormal{if $i$ is odd};\\
\lfloor h/2\rfloor,  & \textnormal{if $i$ is even},
\end{cases}\\
\white_i(h) &= 
\begin{cases}
\lfloor h/2\rfloor,    & \textnormal{if $i$ is odd};\\
\lceil h/2\rceil,  & \textnormal{if $i$ is even}.
\end{cases}   
\end{align*}
Let $w=w_1\cdots w_n$ be a word of length $n$ and let $i\in[n]$. We set
\begin{align*}
\black(w) &= \sum_{i=1}^n \black_i(w_i),\\
\white(w) &= \sum_{i=1}^n \white_i(w_i).
\end{align*}
\end{definition}

We enumerate $k$-ary words and permutations according to their black-white cell capacity. To this end, we introduce two variables, $b$ and $w$, such that to each word $u$ we assign a monomial $b^{\black(u)}w^{\white(u)}$. Summing these monomials over all words of a certain class for a fixed $n$, gives their enumerating polynomial. For example (see Figure \ref{f3}),
\[\sum_{\pi\in S_3} b^{\black(\pi)}w^{\white(\pi)}= 2b^4w^2 + 4b^3w^3.\]

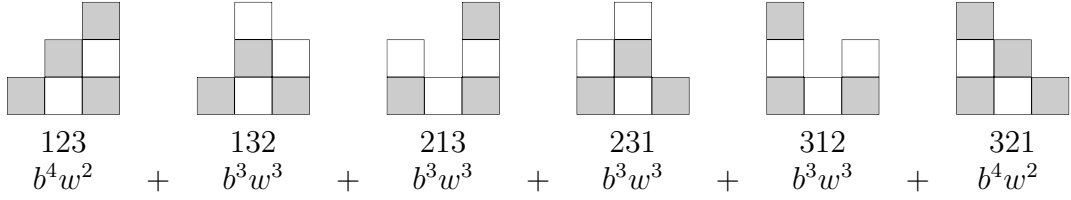
\begin{figure}[H]
\centering
\newcommand{\barperm}[3]{%
\begin{tikzpicture}[scale=0.5, baseline=(current bounding box.south)]
\def\hA{#1}\def\hB{#2}\def\hC{#3}
\def\barpath{(0,0)--(0,\hA)--(1,\hA)--(1,\hB)--(2,\hB)--(2,\hC)--(3,\hC)--(3,0)--cycle}
  \begin{scope}
    \clip \barpath;
    \foreach \x in {0,1,2} {
      \foreach \y in {0,1,2,3} {
        \pgfmathtruncatemacro{\parity}{mod(\x+\y,2)}
        \ifnum\parity=0
          \fill[gray!40] (\x,\y) rectangle ++(1,1); % "black" cells
        \else
          \fill[white]   (\x,\y) rectangle ++(1,1); % "white" cells
        \fi
      }
    }
    \draw[thin] (0,0) grid (3,3);
  \end{scope}
  
\end{tikzpicture}%
}
\begin{tabular}{c c c c c c c c c c c}
\barperm{1}{2}{3} &  & \barperm{1}{3}{2} &  & \barperm{2}{1}{3} &
                  & \barperm{2}{3}{1} &  & \barperm{3}{1}{2} &  & \barperm{3}{2}{1} \\
$123$             &  & $132$             &  & $213$             &
                  & $231$             &  & $312$             &  & $321$             \\
$\;b^{4}w^{2}\;$  & $+$ & $b^{3}w^{3}$ & $+$ & $b^{3}w^{3}$ &
$+$ & $b^{3}w^{3}$ & $+$ & $b^{3}w^{3}$ & $+$ & $b^{4}w^{2}\;$
\end{tabular}
\caption{The six permutations of $[3]$ with their corresponding monomials in the variables $b$ and $w$.}
\label{f3}
\end{figure}

In Theorem \ref{t92} we establish the trivariate generating function for polynomials enumerating $k$-ary words and in Theorem \ref{t93} we establish, for each $n$, the polynomial enumerating permutations.

Of special interest are words having the same number of black and white cells, i.e., words $u$ such that $\black(u)=\white(u)$. We call such words bw-balanced. In Propositions \ref{t401} and \ref{c10} we obtain closed-form formulas for the number of bw-balanced words in $k$-ary words and permutations.

We use the notation $[x^m]\sum_{n\in\mathbb{Z}} a_n x^n = a_m$ for Laurent series coefficients. Vectors are column vectors. The set of real numbers is denoted by $\RR$ and the set of natural numbers $\{1,2,\ldots\}$ by $\NN$.

\section{Main results}

We use the Jacobi polynomials $P_n^{(\alpha,\beta)}(x)$. These are classical orthogonal polynomials with wide applications in mathematical analysis and physics (e.g., \cite[Chapter 4]{Sz} and \cite[Chapter 6.3]{And}). They also appear naturally in combinatorics. In particular, we need the following identity (e.g., \cite[(13.26)]{Gou}).  
\begin{equation}\label{e651}
\sum_{i=0}^{n}\binom{n+\alpha}{i}\binom{n+\beta}{n-i}x^{i}=(x-1)^{n}P_{n}^{(\alpha,\beta)}\left(\frac{x+1}{x-1}\right).    
\end{equation}

\subsection{\texorpdfstring{$k$}{}-ary words}

Let $k\in\NN$ to be used throughout this section and let $n\in\NN$. Denote by $f_n(b,w)$ the polynomial in $b$ and $w$ enumerating the $k$-ary words of length $n$ according to their black-white cell capacity, i.e.,
\[f_n(b,w)=\sum_{u\in[k]^n}b^{\black(u)}w^{\white(u)}.\] Set $f_0(b,w)=1$ and 
let $F_k(x,b,w)$ be the generating function for the $f_n(b,w)$s, i.e., \[F_k(x,b,w)= \sum_{n\geq0}f_n(b,w)x^n.\]

\begin{theorem}\label{t92}
We have \[F_k(x,b,w)=\frac{1+xg_k(b,w)}{1-x^{2}g_k(b,w)g_k(w,b)},\] where \[g_k(b,w)=\frac{b(1-(bw)^{\lceil k/2\rceil })+bw(1-(bw)^{\lfloor k/2\rfloor })}{1-bw}.\]
\end{theorem}

\begin{proof}
Consider a letter $h\in[k]$ corresponding to an index $i\in[n]$. If $i$ is odd, then the letter contributes $b^{\left\lceil h/2\right\rceil }w^{\left\lfloor h/2\right\rfloor }$ to $x^n$ and if $i$ is even, its contribution is $b^{\left\lfloor  h/2\right\rfloor }w^{\left\lceil h/2\right\rceil }$. Let $g_k(b,w)$ stand for the contribution of all possible letters at an odd index. Thus, \[g_k(b,w)=\sum_{h=1}^{k}b^{\lceil h/2\rceil }w^{\lfloor h/2\rfloor }=\frac{b(1-(bw)^{\lceil k/2\rceil })+bw(1-(bw)^{\lfloor k/2\rfloor })}{1-bw}.\] Notice that $g_k(w,b)$ corresponds to the contribution of all possible letters at an even index. Since letters at different indices are independent, their joint contribution is obtained by multiplication. Thus, \[f_n(b,w) = \begin{cases}
(g_k(b,w)g_k(w,b))^{m}, & \text{if }n=2m;\\
(g_k(b,w)g_k(w,b))^{m}g_k(b,w), & \text{if }n=2m+1.
\end{cases}\] It follows that 
\begin{align}
F_{k}(x,b,w)	&=\sum_{m\geq0}(g_k(b,w)g_k(w,b))^{m}x^{2m}+\sum_{m\geq0}(g_k(b,w)g_k(w,b))^{m}g_k(b,w)x^{2m+1}\nonumber\\
&=(1+xg_k(b,w))\sum_{m\geq0}(x^{2}g_k(b,w)g_k(w,b))^{m}\nonumber\\
&=\frac{1+xg_k(b,w)}{1-x^{2}g_k(b,w)g_k(w,b)}.\nonumber \qedhere
\end{align} 
\end{proof}

\begin{proposition}\label{t401}
Denote by $\textnormal{bal}_k(n)$ the number of bw-balanced $k$-ary words of length $n$ and let $\textnormal{BAL}_k(x)=\sum_{n\geq0}\textnormal{bal}_k(n)x^n$ be the corresponding generating function.  
Then $\textnormal{BAL}_1(x) = 1/(1-x^2)$ and, for $k\geq 2$,
\begin{equation}\label{e871}
\textnormal{BAL}_k(x) =\frac{1}{\Delta_k}\left(1+\lfloor k/2\rfloor x+\frac{1-\left(\lfloor k/2\rfloor ^{2}+\lceil k/2\rceil ^{2}\right)x^{2}-\Delta_k}{2\lfloor k/2\rfloor x}\right),
\end{equation} where 
\[\Delta_k=\sqrt{\left(1-\left(\lfloor k/2\rfloor ^{2}+\lceil k/2\rceil ^{2}\right)x^{2}\right)^{2}-4\lfloor k/2\rfloor ^{2}\lceil k/2\rceil ^{2}x^{4}}.\]
Furthermore, let $\alpha=\lceil n/2\rceil-\lfloor n/2\rfloor$. Then, for every $k\geq 1$,
\begin{equation}\label{e819}
\textnormal{bal}_k(n)=
\begin{cases}
\left(\frac{k}{2}\right)^n\binom{n}{\lfloor n/2\rfloor},&\textnormal{if $k$ is even}; \\ 
\left(\frac{k-1}{2}\right)^{\alpha}k^{\lfloor n/2\rfloor}P_{\lfloor n/2\rfloor}^{(\alpha,0)}\left(\frac{k^{2}+1}{2k}\right), &\textnormal{if $k$ is odd}. 
\end{cases}
\end{equation}
\end{proposition}

\begin{proof}
If $k=1$, there is only one word for each $n$, namely, $1\cdots 1$. This word is bw-balanced if and only if $n$ is even. The corresponding generating function is then $\textnormal{BAL}_1(x)=1/(1-x^2)$ and it is easy to see that \eqref{e819} holds true also in this case.

Assume now that $k\geq 2$ and set 
\begin{align*}
A&=1-(\lfloor k/2\rfloor^2+\lceil k/2\rceil^2)x^2,\\
B&= \lfloor k/2\rfloor \lceil k/2\rceil x^2.
\end{align*} Notice that $\Delta_k = \sqrt{A^2-4B^2}$. Let $\sum_{n\in\mathbb{Z}} c_n t^n$ be the Laurent series of $1/(A-B(t+t^{-1}))$. It is not hard to see that \[c_0=\frac{1}{\Delta_k},\qquad c_{-1} =\frac{A-\Delta_k}{2B\Delta_k}.\] Thus, 
\begin{align*}
\textnormal{BAL}_k(x)&=[t^0]F_k(x,t,t^{-1})\\
&=[t^0]\frac{1+\lfloor k/2\rfloor x+\lceil k/2\rceil xt}{A-B\left(t+t^{-1}\right)}\\
&=\frac{1+\lfloor k/2\rfloor x}{\Delta_k} + \frac{A-\Delta_k}{2\lfloor k/2\rfloor x\Delta_k},
\end{align*} proving \eqref{e871}.

We now wish to prove \eqref{e819}. For a letter $h\in[k]$ at index $i\in[n]$ we have
\[
\black_i(h)-\white_i(h)=
\begin{cases}
+1,& \textnormal{if $h$ is odd and $i$ is odd};\\
-1,& \textnormal{if $h$ is odd and $i$ is even};\\
0,& \textnormal{if $h$ is even}.
\end{cases}
\]
Thus, a word is bw-balanced if and only if the number of odd letters at odd indices is equal to the number of odd letters at even indices. Let $r$ be this common number. Clearly, $r\in\{0,1,\ldots,\lfloor n/2\rfloor\}$. There are $\binom{\lceil n/2\rceil}{r}$ ways to choose which of the letters at odd indices will be odd, and $\binom{\lfloor n/2\rfloor}{r}$ ways to choose which of the letters at even indices will be odd. For each of the $2r$ odd letters, there are $\lceil k/2\rceil$ possibilities. For each of the rest $n-2r$ letters, which are even, there are
$\lfloor k/2\rfloor$ possibilities. Thus, the number of bw-balanced words for this $r$ is 
\begin{equation}\label{e104}
\binom{\lceil n/2\rceil}{r}\binom{\lfloor n/2\rfloor}{r}
\lceil k/2\rceil^{2r}
\lfloor k/2\rfloor^{n-2r}.
\end{equation}
Summing \eqref{e104} over all possible values for $r$ gives 
\begin{equation}\label{e931}
\textnormal{bal}_k(n)=\sum_{r=0}^{\lfloor n/2\rfloor}
\binom{\lceil n/2\rceil}{r}\binom{\lfloor n/2\rfloor}{r}
\lceil k/2\rceil^{2r}
\lfloor k/2\rfloor^{n-2r}.
\end{equation} Suppose that $k$ is even. Then
$\lceil k/2\rceil=\lfloor k/2\rfloor=k/2$. Using Vandermonde's identity, we have
\[\textnormal{bal}_k(n)=
(k/2)^n\sum_{r=0}^{\lfloor n/2\rfloor}\binom{\lceil n/2\rceil}{r}\binom{\lfloor n/2\rfloor}{r}=(k/2)^n\binom{n}{\lfloor n/2\rfloor}.
\]
Assume now that $k$ is odd and let 
\[q = \left(\frac{\lceil k/2\rceil}{\lfloor k/2\rfloor}\right)^{2}=\left(\frac{k+1}{k-1}\right)^{2}.\] By \eqref{e931}, 
\begin{align*}
\textnormal{bal}_{k}(n)&=\left(\frac{k-1}{2}\right)^n\sum_{r=0}^{\lfloor n/2\rfloor}\binom{\lfloor n/2\rfloor+\alpha}{r}\binom{\lfloor n/2\rfloor}{r}q^{r}\\
&=\left(\frac{k-1}{2}\right)^n(q-1)^{\lfloor n/2\rfloor}P_{\lfloor n/2\rfloor}^{(\alpha,0)}\left(\frac{q+1}{q-1}\right)\\
&=\left(\frac{k-1}{2}\right)^n\frac{(4k)^{\lfloor n/2\rfloor}}{(k-1)^{2\lfloor n/2\rfloor}}P_{\lfloor n/2\rfloor}^{(\alpha ,0)}\left(\frac{k^{2}+1}{2k}\right)\\&=\left(\frac{k-1}{2}\right)^{\alpha}k^{\lfloor n/2\rfloor}P_{\lfloor n/2\rfloor}^{(\alpha,0)}\left(\frac{k^{2}+1}{2k}\right),
\end{align*} where in the second equality we used \eqref{e651}.
\end{proof}

In Table~\ref{tab1} below we list the initial values of $\textnormal{bal}_k(n)$, for $k=1,\ldots,6$ and $n=0,\ldots,10$.

\begin{table}[H]
\centering
\small
\setlength{\tabcolsep}{5pt}
\renewcommand{\arraystretch}{1.15}
\begin{tabular}{c|rrrrrrrrrrr}
\hline
$k \backslash n$ & 0 & 1 & 2 & 3 & 4 & 5 & 6 & 7 & 8 & 9 & 10 \\
\hline
1 & 1 & 0 & 1 & 0 & 1 & 0 & 1 & 0 & 1 & 0 & 1 \\
2 & 1 & 1 & 2 & 3 & 6 & 10 & 20 & 35 & 70 & 126 & 252 \\
3 & 1 & 1 & 5 & 9 & 33 & 73 & 245 & 593 & 1921 & 4881 & 15525 \\
4 & 1 & 2 & 8 & 24 & 96 & 320 & 1280 & 4480 & 17920 & 64512 & 258048 \\
5 & 1 & 2 & 13 & 44 & 241 & 950 & 5005 & 21080 & 109345 & 477962 & 2458573 \\
6 & 1 & 3 & 18 & 81 & 486 & 2430 & 14580 & 76545 & 459270 & 2480058 & 14880348 \\
\hline
\end{tabular}
\caption{Number of bw-balanced $k$-ary words of length $n$ for $k=1,\dots,6$ and $n=0,\dots,10$.}
\label{tab1}
\end{table}

\begin{remark}
Row $k=2$ in Table \ref{tab1} corresponds to \seqnum{A001405}. Row $k=3$ coincides with \seqnum{A084771}, but only for even $n$. Row $k=4$ coincides with \seqnum{A060899}, which is defined only for even $n$. Nevertheless, the lattice–path interpretations of the latter two sequences do not seem to indicate a more general connection between bw-balanced $k$-ary words and lattice paths.
\end{remark}

We now establish asymptotic proportion of bw-balanced $k$-ary words. 

\begin{corollary}
Assume that $k\geq 2$. Then
\[
\frac{\mathrm{bal}_k(n)}{k^n}\sim\sqrt{\frac{2}{\pi n}}\times 
\begin{cases}
1, & \textnormal{if $k$ is even};\\
\sqrt{\frac{k^2}{k^2-1}}, & \textnormal{if $k$ is odd}.
\end{cases}
\]
\end{corollary}

\begin{proof}
First assume that $k$ is even. By \eqref{e819} and using Stirling's formula, 
\[
\frac{\textnormal{bal}_k(n)}{k^n}=
\frac{1}{2^n}\binom{n}{\lfloor n/2\rfloor}
\sim\frac{1}{2^n}\sqrt{\frac{2}{\pi n}}2^{n}=\sqrt{\frac{2}{\pi n}}.
\] Assume now that $k$ is odd. By \eqref{e819},
\begin{equation}\label{e820}
\textnormal{bal}_k(n)=
\left(\frac{k-1}{2}\right)^{\alpha}k^{\lfloor n/2\rfloor}P_{\lfloor n/2\rfloor}^{(\alpha,0)}\left(\frac{k^{2}+1}{2k}\right).
\end{equation} 
By \cite[(8.21.9)]{Sz}, for $x\notin[-1,1]$, arbitrary $\alpha,\beta\in\RR$, and large $m\in\NN$, 
\begin{equation}\label{e734} 
P_{m}^{(\alpha,\beta)}(x)\sim\frac{(\sqrt{x+1}+\sqrt{x-1})^{\alpha+\beta}(x+\sqrt{x^{2}-1})^{m+1/2}}{\sqrt{2\pi m}(x^{2}-1)^{1/4}\sqrt{(x-1)^{\alpha}}\sqrt{(x+1)^{\beta}}}.
\end{equation}
Set $\beta=0$ and let $x=(k^2+1)/2k$. Then $x>1$ and 
\begin{align*}
x+\sqrt{x^2-1}&=k,& \sqrt{x+1}+\sqrt{x-1}&=\sqrt{2k},\\
\sqrt{(x-1)^{\alpha}}&=\left(\frac{k-1}{\sqrt{2k}}\right)^{\alpha},
&(x^2-1)^{1/4}&=\sqrt{\frac{k^2-1}{2k}}.
\end{align*} Thus, \eqref{e734} reduces to
\begin{equation}\label{ez1}
P_m^{(\alpha,0)}\left(\frac{k^2+1}{2k}\right)
\sim \frac{k^m}{\sqrt{2\pi m}}
\sqrt{\frac{2k^{2}}{k^{2}-1}}
\Big(\frac{2k}{k-1}\Big)^{\alpha}.
\end{equation} Set $\alpha=\lceil n/2\rceil - \lfloor n/2\rfloor$ and notice that $n=2\lfloor n/2\rfloor+\alpha$. It follows from \eqref{e820} and \eqref{ez1} that
\begin{align*}    
\mathrm{bal}_k(n) &\sim \left(\frac{k-1}{2}\right)^{\alpha}\Big(\frac{2k}{k-1}\Big)^{\alpha}\frac{k^{2\lfloor n/2\rfloor}}{\sqrt{2\pi\lfloor n/2\rfloor}}\sqrt{\frac{2k^{2}}{k^{2}-1}}\\
&=\frac{k^{2\lfloor n/2\rfloor+\alpha}}{\sqrt{2\pi\lfloor n/2\rfloor}}\sqrt{\frac{2k^{2}}{k^{2}-1}}\\
&=\frac{k^n}{\sqrt{2\pi\lfloor n/2\rfloor}}\sqrt{\frac{2k^{2}}{k^{2}-1}}.
\end{align*}
Dividing both sides by $k^n$ and using $\lfloor n/2\rfloor\sim n/2$, the assertion follows.
\end{proof}

\subsection{Permutations}

For $n\geq 1$ denote by $\mathbf{1}_n$ the all-ones vector of length $n$. We shall need the following result, concerned with the permanent of a special kind of matrix. The reader is referred to  \cite[Chapter 7]{Br} for general properties of the permanent.

\begin{lemma}\label{l1}
Let $n\geq 1$ and let $0\leq m\leq n$. Let $v=(v_1,\dots,v_n)^T\in\RR^n$ and let $A$ be a square matrix of size $n$ with $m$ columns equal to $\mathbf{1}_n$, and the rest of the columns equal to $v$. Then
\[
\textnormal{perm}(A)
= m!(n-m)!\sum_{\substack{R\subseteq[n]\\ |R|=n-m}}\ \prod_{i\in R} v_i.
\]
\end{lemma}

\begin{proof}
By the invariance of the permanent under permutations of the columns, we may assume that the $j$th column of $A$ is $\mathbf{1}_n$, for each $1\leq j\leq m$, and is $v$, for each $m+1\leq j\leq n$. We have
\begin{align*}
\textnormal{perm}(A)&=\sum_{\pi\in S_n}\prod_{i\in[n]}A_{i,\pi_i}\\
&=\sum_{\pi\in S_n}\left(\prod_{\substack{i\in [n]\\ 1\leq \pi_i\leq m}}1\right)\left(\prod_{\substack{i\in [n]\\ m+1\leq \pi_i\leq n}}v_i\right)\\
&=\sum_{\substack{R\subseteq[n]\\ |R|=n-m}} \ \sum_{\substack{\pi\in S_n \textnormal{ with}\\ m+1\leq \pi_i\leq n,\;\forall i\in R}}\ \prod_{i\in R} v_i\\
&=m!(n-m)!\sum_{\substack{R\subseteq[n]\\ |R|=n-m}}\ \prod_{i\in R} v_i,
\end{align*} as claimed.
\end{proof}

Denote by $f_n(b,w)$ the polynomial in $b$ and $w$ enumerating the permutations of $[n]$ according to their black-white cell capacity, i.e.,
\[f_n(b,w)=\sum_{\pi\in S_n}b^{\black(\pi)}w^{\white(\pi)}.\]

\begin{theorem}\label{t93}
Set $\alpha=\lceil n/2\rceil-\lfloor n/2\rfloor$. Then
\begin{equation}\label{r4}
f_n(b,w)
=(bw)^{\lfloor n^2/4\rfloor}\lfloor n/2\rfloor!\lceil n/2\rceil!b^{\alpha}(w-b)^{\lfloor n/2\rfloor}P_{\lfloor n/2\rfloor}^{(\alpha,0)}\left(\frac{w+b}{w-b}\right).
\end{equation}
\end{theorem}

\begin{proof}
Let $M$ be the square matrix of size $n$ defined by
\[
M_{ij}=
\begin{cases}
b^{\lceil j/2\rceil}w^{\lfloor j/2\rfloor}, & \textnormal{if $i$ is odd},\\
b^{\lfloor j/2\rfloor}w^{\lceil j/2\rceil}, & \textnormal{if $i$ is even}.
\end{cases}
\] Clearly, \[f_n(b,w)=\displaystyle\sum_{\pi\in S_n}\prod_{i=1}^nb^{\black_i(\pi_i)}w^{\white_i(\pi_i)}=\textnormal{perm}(M).\]  Now, \[M_{ij}=(bw)^{\lfloor j/2\rfloor}\times\begin{cases}
1, & \textnormal{if $j$ is even};\\
b, & \textnormal{if $j$ is odd and $i$ is odd};\\
w, & \textnormal{if $j$ is odd and $i$ is even}.
\end{cases}\] Let $v=(b,w,b,\ldots)^T$. Thus, the $j$th column of $M$ is given by
 \[
(bw)^{\lfloor j/2\rfloor }\times\begin{cases}
\mathbf{1}_n, & \textnormal{if \ensuremath{j} is even};\\
v, & \textnormal{if \ensuremath{j} is odd}.
\end{cases}
\] The product of the column coefficients is given by
\begin{align}
\prod_{j=1}^n(bw)^{\lfloor j/2\rfloor}&=
\left(\prod_{m=1}^{\lfloor n/2\rfloor} (bw)^m\right)  \left(\prod_{m=0}^{\lceil n/2\rceil-1} (bw)^m\right)\nonumber\\
&=(bw)^{\lfloor n/2\rfloor(\lfloor n/2\rfloor+1)/2+ \lceil n/2\rceil(\lceil n/2\rceil-1)/2}\nonumber\\
&=(bw)^{\lfloor n^2/4\rfloor}\label{e0}.
\end{align} Let $A$ be the square matrix of size $n$ whose $j$th column is given by 
\[\begin{cases}
\mathbf{1}_n, & \textnormal{if \ensuremath{j} is even};\\
v, & \textnormal{if \ensuremath{j} is odd}.
\end{cases}\] Applying Lemma \ref{l1} on $A$ with $m= \lfloor n/2\rfloor$ yields
\begin{equation}\label{e2}
\textnormal{perm}(A)
=\lfloor n/2\rfloor!\lceil n/2\rceil!\sum_{\substack{R\subseteq[n]\\ |R|=\lceil n/2\rceil}}\ \prod_{i\in R} v_i.
\end{equation}
Now, the product $\prod_{i\in R}v_i$ depends only on the number $r$ of even (or, equivalently, odd) numbers in $R$. More precisely, 
\begin{equation}\label{e31}
\prod_{i\in R}v_i=b^{\lceil n/2\rceil-r}w^r.
\end{equation} The set $[n]$ consists of $\lceil n/2\rceil$ odd numbers and of $\lfloor n/2\rfloor$ even numbers. Thus, $0\leq r\leq \lfloor n/2 \rfloor$. The number of subsets of $[n]$ of size $\lceil n/2 \rceil$ consisting of $r$ even numbers (and hence of $\lceil n/2\rceil-r$ odd numbers) is given by
\begin{equation}\label{e41}
\binom{\lfloor n/2\rfloor}{r}\binom{\lceil n/2\rceil}{\lceil n/2\rceil-r}=\binom{\lfloor n/2\rfloor}{r}\binom{\lceil n/2\rceil}{r}.
\end{equation} It follows from \eqref{e2}, \eqref{e31}, and \eqref{e41} that
\begin{align*}
\textnormal{perm}(A)&=\lfloor n/2\rfloor!\lceil n/2\rceil!\sum_{r=0}^{\lfloor n/2\rfloor}\binom{\lfloor n/2\rfloor}{r}\binom{\lceil n/2\rceil}{r}b^{\lceil n/2\rceil-r}w^{r}\\&=\lfloor n/2\rfloor!\lceil n/2\rceil!b^{\lceil n/2\rceil}\sum_{r=0}^{\lfloor n/2\rfloor}\binom{\lfloor n/2\rfloor+\alpha}{r}\binom{\lfloor n/2\rfloor}{r}\left(\frac{w}{b}\right)^{r}\\&=\lfloor n/2\rfloor!\lceil n/2\rceil!b^{\lceil n/2\rceil}\left(\frac{w}{b}-1\right)^{\lfloor n/2\rfloor}P_{\lfloor n/2\rfloor}^{(\alpha,0)}\left(\frac{\frac{w}{b}+1}{\frac{w}{b}-1}\right)\\&=\lfloor n/2\rfloor!\lceil n/2\rceil!b^{\alpha}(w-b)^{\lfloor n/2\rfloor}P_{\lfloor n/2\rfloor}^{(\alpha,0)}\left(\frac{w+b}{w-b}\right).
\end{align*}
Using \eqref{e0}, we see that $\textnormal{perm}(M)=(bw)^{\lfloor n^2/4\rfloor}\textnormal{perm}(A)$ and the assertion follows.
\end{proof}

In the following result we use standard notation for shifted factorials and hypergeometric series (e.g., \cite[(1.1.2) and (2.1.2)]{And}).

\begin{proposition}\label{c10}
Denote by $\textnormal{bal}_{S_n}(n)$ the number of bw-balanced permutations of $[n]$ and let $\textnormal{BAL}_{S_n}(x)=\sum_{n\geq0}\frac{\textnormal{bal}_{S_n}(n)}{n!}x^n$ be the corresponding exponential generating function. Then 
\[\textnormal{bal}_{S_n}(n) = 
\begin{cases}  
\lfloor n/2\rfloor!\lceil n/2\rceil!
\binom{\lceil n/2\rceil}{\lceil n/2\rceil/2}\binom{\lfloor n/2\rfloor}{\lceil n/2\rceil/2},&\textnormal{if $n\equiv0,3\pmod4$};\\
0,&\textnormal{if $n\equiv1,2\pmod4$}.
\end{cases}\]
Furthermore, 
\[
\textnormal{BAL}_{S_n}(x)=\frac{(1+x)G(x)-1}{x},
\] where \[G(x)=\rFq{3}{2}{\frac12, \frac12, \frac12}{\frac14,\frac34}{x^{4}}.
\]
\end{proposition}

\begin{proof}
Substituting $b=t$ and $w=t^{-1}$ in \eqref{r4}, we have
\[\textnormal{bal}_{S_n}(n) = [t^0]f_{n}(t,t^{-1})=[t^0]\left(\lfloor n/2\rfloor!\lceil n/2\rceil!\sum_{r=0}^{\lfloor n/2\rfloor}\binom{\lceil n/2\rceil}{r}\binom{\lfloor n/2\rfloor}{r}t^{\lceil n/2\rceil-2r}\right).
\] Clearly, if $\lceil n/2\rceil$ is odd, then $\lceil n/2\rceil-2r\neq 0$ for every $r$ and hence $\textnormal{bal}_{S_n}(n)=0$. If $\lceil n/2\rceil$ is even, then $\lceil n/2\rceil-2r= 0$ if and only if $r=\lceil n/2\rceil/2$ and therefore \[\textnormal{bal}_{S_n}(n)=\lfloor n/2\rfloor!\lceil n/2\rceil!
\binom{\lceil n/2\rceil}{\lceil n/2\rceil/2}\binom{\lfloor n/2\rfloor}{\lceil n/2\rceil/2}.\] Regarding the egf, notice that $\textnormal{bal}_{S_n}(n) \neq 0$ if and only if $n\equiv0,3\pmod4$. Thus, 
\begin{equation}\label{eo2}
\textnormal{BAL}_{S_n}(x) = \sum_{m\geq 0}\frac{\textnormal{bal}_{S_n}(4m)}{(4m)!}x^{4m} + \sum_{m\geq 1}\frac{\textnormal{bal}_{S_n}(4m-1)}{(4m-1)!}x^{4m-1}.
\end{equation}
Now, using the duplication identity
\((a)_{2m}=2^{2m}\left(\frac {a}{2}\right)_m\left(\frac{a+1}{2}\right)_m\) (e.g., \cite[p.\ 22]{And}), we have
\[
\frac{\textnormal{bal}_{S_n}(4m)}{(4m)!}
=\frac{(2m)!^4}{m!^4(4m)!}
=\frac{(\frac{1}{2})_m^3}{m!(\frac{1}{4})_m(\frac{3}{4})_m}.
\]
Hence,
\begin{equation}\label{eo3}
\sum_{m\geq 0}\frac{\textnormal{bal}_{S_n}(4m)}{(4m)!}x^{4m}=G(x).
\end{equation} Similarly,
\begin{equation}\label{eo4}
\frac{\textnormal{bal}_{S_n}(4m-1)}{(4m-1)!}
=\frac{(2m)!^2(2m-1)!^2}{(m!)^3(m-1)!(4m-1)!}
=\frac{(\frac{1}{2})_m^3}{m!(\frac{1}{4})_m(\frac{3}{4})_m}=\frac{\textnormal{bal}_{S_n}(4m)}{(4m)!}.
\end{equation} It follows that
\[
\sum_{m\ge1}\frac{\textnormal{bal}_{S_n}(4m-1)}{(4m-1)!}x^{4m-1}
=\frac{G(x)-1}{x}.
\]
Finally, by \eqref{eo2}, \eqref{eo3}, and \eqref{eo4}, \[\textnormal{BAL}_{S_n}(x)= G(x)+\frac{G(x)-1}{x} =\frac{(1+x)G(x)-1}{x},\] as asserted.
\end{proof}

\begin{definition}
The number of odd displacements of a permutation $\pi\in S_n$, denoted by $\oD(\pi)$, is defined to be
\[
\oD(\pi)=|\{i\in[n]: i-\pi_i\ \textnormal{is odd}\}|.
\]
\end{definition}

\begin{theorem}\label{t3}
Let $\pi\in S_n$. Then $\pi$ is bw-balanced if and only if $\oD(\pi)=\lceil n/2\rceil$.
\end{theorem}

\begin{proof}
Set
\begin{align*}
o&=|\{i\in[n] \;:\; \pi_i \textnormal{ is odd and } i \textnormal{ is odd}\}|,\\
e&=|\{i\in[n]\;:\; \pi_i \textnormal{ is odd and } i \textnormal{ is even}\}|.
\end{align*}
Let $i\in[n]$. Then
\[
\black_i(\pi_i)-\white_i(\pi_i)=
\begin{cases}
+1,& \textnormal{if $\pi_i$ is odd and $i$ is odd};\\
-1,& \textnormal{if $\pi_i$ is odd and $i$ is even};\\
0,& \textnormal{if $\pi_i$ is even}.
\end{cases}
\] Thus, the total difference between the numbers of black and white cells in $\pi$ is given by \[\sum_{i=1}^n (\black_i(\pi_i)-\white_i(\pi_i))=o-e.\] Hence, $\pi$ is bw-balanced if and only if $o=e$. On the other hand, 
\begin{align*}
\oD(\pi)&=
|\{i\in[n]\;:\; i \text{ is even and } \pi_i \text{ is odd}\}| + |\{i\in[n]\;:\; i \text{ is odd and } \pi_i \text{ is even}\}|\\
&= e +(\lceil n/2\rceil - o)\\
&=\lceil n/2\rceil - (o-e).
\end{align*} Hence, $\oD(\pi)=\lceil n/2\rceil$ if and only if $o=e$.
\end{proof}

\begin{remark}
Denote by $T(n,m)$ the number of permutations of $[n]$ with exactly $2m-2$ odd displacements.
By \cite[\seqnum{A226288}]{Sl},
\[
T(n,m)=\lfloor n/2\rfloor!\lceil n/2\rceil!
\binom{\lfloor n/2\rfloor}{m-1}\binom{\lceil n/2\rceil}{m-1}.
\] Thus, the number of permutations of $[n]$ with exactly $\lceil n/2\rceil$ odd displacements is given by
\[
T(n,\lceil n/2\rceil/2+1)=\lfloor n/2\rfloor!\lceil n/2\rceil!
\binom{\lfloor n/2\rfloor}{\lceil n/2\rceil/2}\binom{\lceil n/2\rceil}{\lceil n/2\rceil/2}, 
\] exactly our formula for $\textnormal{bal}_{S_n}(n)$ in Proposition \ref{c10}.
\end{remark}

Standard application of Stirling’s formula yields the following asymptotic proportion of bw-balanced permutations. 

\begin{corollary}
For $n\equiv0,3\pmod4$, we have
\[
\frac{\textnormal{bal}_{S_n}(n)}{n!} \sim  \sqrt{\frac{8}{\pi n}}.
\]
\end{corollary}

\end{document}